\documentclass[12pt, reqno]{amsart}
\usepackage{pgf,tikz}

\usepackage{amsfonts,amsmath,amsthm,amssymb,pgf,tikz,verbatim}
\usepackage{times}
\usepackage{fullpage}
\usetikzlibrary{arrows}
\begin{document}
\definecolor{qqqqff}{rgb}{0,0,0}

\newtheorem{theorem}{Theorem}[section]
\newtheorem{lemma}[theorem]{Lemma}
\newtheorem{coro}[theorem]{Corollary}
\newtheorem{definition}[theorem]{Definition}
\newtheorem{example}[theorem]{Example}
\newtheorem{xca}[theorem]{Exercise}
\newtheorem{claim}{Claim}
\theoremstyle{remark}
\newtheorem{remark}[theorem]{Remark}

\numberwithin{equation}{section}
\newtheorem{cj}[theorem]{Conjecture}

\newtheorem{defi}[theorem]{Definition}
\newtheorem{prop}[theorem]{Proposition}

\def\binom#1#2{{#1}\choose{#2}}

\def\slfrac#1#2{\hbox{\kern.1em %
 \raise.5ex\hbox{\the\scriptfont0 #1}\kern-.11em %
 /\kern-.15em\lower.25ex\hbox{\the\scriptfont0 #2}}}

\font\phvr=phvr at 10pt
\newcommand\he[1]{\mbox{\phvr  #1}}

\newcommand{\fot}{\frac{1}{2}}
\newcommand{\eqn}[1]{(\ref{#1})}
\newcommand{\hsp}{\hspace*{\parindent}}
\newcommand{\vsp}{\vspace{.1in}}

\newcommand{\bA}{\bar{A}}

\newcommand{\Latilde}{\tilde{\La}}
\newcommand{\determ}{\rm det}
\newcommand{\La}{\Lambda}
\newcommand{\la}{\lambda}
\newcommand{\af}{\alpha}
\newcommand{\ep}{\epsilon}
\newcommand{\eps}{\epsilon}
\newcommand{\om}{\omega}
\newcommand{\Om}{\Omega}

\newcommand{\agh}{{a}}

\newcommand{\AAA}{{\mathbb A}}
\newcommand{\ZZ}{{\mathbb Z}}
\newcommand{\RR}{{\mathbb R}}
\newcommand{\GG}{{\mathbb G}}
\newcommand{\HH}{{\mathbb H}}
\newcommand{\LL}{{\mathbb L}}
\newcommand{\NN}{{\mathbb N}}
\newcommand{\FF}{{\mathbb F}}
\newcommand{\PP}{{\mathbb P}}
\newcommand{\QQ}{{\mathbb Q}}
\newcommand{\TT}{{\mathbb T}}
\newcommand{\CC}{{\mathbb C}}

\newcommand{\bb}{{\bf b}}
\newcommand{\bc}{{\bf c}}
\newcommand{\bd}{{\bf d}}
\newcommand{\be}{{\bf e}}
\newcommand{\bk}{{\bf k}}
\newcommand{\bD}{{\bf D}}
\newcommand{\bF}{{\bf F}}
\newcommand{\bG}{{\bf G}}
\newcommand{\bU}{{\bf U}}
\newcommand{\bm}{{\bf m}}
\newcommand{\bz}{{\bf z}}
\newcommand{\bv}{{\bf v}}
\newcommand{\bw}{{\bf w}}
\newcommand{\bx}{{\bf x}}
\newcommand{\bh}{{\bf h}}

\newcommand{\bbS}{\Sigma}
\newcommand{\sB}{{\mathcal B}}
\newcommand{\sC}{{\mathcal C}}
\newcommand{\sD}{{\mathcal D}}
\newcommand{\sE}{{\mathcal E}}
\newcommand{\sF}{{\mathcal F}}
\newcommand{\sG}{{\mathcal G}}
\newcommand{\sO}{{\mathcal O}}
\newcommand{\EG}{{\delta}}

\newcommand{\sH}{{\mathcal H}}
\newcommand{\sL}{{\mathcal L}}
\newcommand{\sM}{{\mathcal M}}
\newcommand{\sN}{{\mathcal N}}
\newcommand{\sP}{{\mathcal P}}
\newcommand{\sR}{{\mathcal R}}
\newcommand{\sS}{{\mathcal S}}
\newcommand{\sT}{{\mathcal T}}
\newcommand{\sW}{{\mathcal W}}
\newcommand{\sX}{{\mathcal X}}
\newcommand{\sY}{{\mathcal Y}}
\newcommand{\sZ}{{\mathcal Z}}
\newcommand{\sLP}{{\mathcal L}{\mathcal P}}
\newcommand{\sEP}{{\mathcal E}{\mathcal P}}
\newcommand{\sGP}{{\mathcal G}{\mathcal P}}
\newcommand{\sHP}{{\mathcal H}{\mathcal P}}

\newcommand{\blambda}{{\bf \lambda}}
\newcommand{\bE}{{\bf E}}
\newcommand{\Ein}{{\rm Ein}}
\newcommand{\Eone}{{\rm E}_1}
\newcommand{\Li}{\,{\rm Li}}
\newcommand{\lf}{\lfloor}
\newcommand{\rf}{\rfloor}
\newcommand{\lc}{\lceil}
\newcommand{\rc}{\rceil}
\newcommand{\meas}{{ \rm meas}}

\renewcommand{\theequation}{\arabic{section}.\arabic{equation}}

\newcommand{\aext}{\measuredangle_{\operatorname{ext}}}

\title{$3x+1$ inverse orbit generating functions almost always have natural boundaries}
\author{Jason P. Bell}
\address{Dept. of Mathematics,
University of Waterloo,
Waterloo, ON }
\curraddr{}
\email{jpbell@uwaterloo.ca}

\author{Jeffrey C. Lagarias}
\address{Dept. of Mathematics,
University of Michigan,
Ann Arbor, MI 48109-1043, USA}
\curraddr{}
\email{lagarias@umich.edu}

\thanks{The research  of the second author was partially supported by NSF Grants DMS-1101373 and DMS-1401224.}

\date{August 22, 2014, v4.3}
\maketitle

\begin{abstract}
The $3x+k$ function $T_{k}(n)$ sends $n$ to  $(3n+k)/2$ resp. $n/2,$
according as $n$ is odd, resp.  even, where $k \equiv \pm 1~(\bmod \, 6)$. 
The map $T_k(\cdot)$ sends integers to integers, and for $m \ge 1$
let $n \rightarrow m$
mean that $m$ is in the forward orbit of $n$ under iteration of $T_k(\cdot).$
We consider the generating functions $f_{k,m}(z) = \sum_{n>0, n \rightarrow m}  z^{n},$
which are holomorphic in the unit disk. 
We give sufficient conditions on $(k,m)$ for the functions  $f_{k, m}(z)$ 
have the unit circle $\{|z|=1\}$ as a natural boundary to analytic continuation.
For the  $3x+1$ function these conditions hold for all  $m \ge 1$ to show
that $f_{1,m}(z)$ has the unit circle as a natural boundary
except possibly for $m= 1, 2, 4$ and $8$. The $3x+1$ Conjecture  is equivalent to the assertion 
that $f_{1, m}(z)$ is  a rational function of $z$ for  the remaining values $m=1,2, 4, 8$.
\end{abstract}

%
%
%
%

\section{Introduction}

The $3x+1$ function is given by 
\begin{equation}\label{101}
T(n) = T_1(n) := \begin{cases}
\frac{3n+1}{2} & \text{if $n$ is odd},\\
\frac{n}{2} & \text{if $n$ is even}.
\end{cases}
\end{equation}
The  {\em  $3x+1$ problem} (or {\em Collatz problem}) concerns
the behavior of this map under  iteration,
 restricted to the domain  of positive integers $\NN^{+}$. 
This domain  is invariant under iteration, and it contains the 
periodic orbit  $\{ 1, 2\}$ of $T$, which is 
 the only  periodic orbit known on $\NN^{+}$ at present.  The   {\em $3x+1$ Conjecture}
(or  {\em Collatz Conjecture}) asserts that  every positive
integer under iteration enters this periodic orbit.
The $3x+1$ conjecture appears to be intractable at present, see for example 
\cite{Lag85, Wir98} and for recent viewpoints \cite{Lag10} and \cite{Con13}.

The $3x-1$ function is given by 
\begin{equation}\label{101}
T_{-1} (n) := \begin{cases}
\frac{3n-1}{2} & \text{if $n$ is odd};\\
\frac{n}{2} & \text{if $n$ is even}.
\end{cases}
\end{equation}
It satisfies $T_{-1}(n) = - T_1 (-n).$ There is an analogous $3x-1$ problem concerning 
its behavior under iteration on the positive integers $\NN^{+}$, which  has recently been
studied by Berg and Opfer \cite{BO13}. This function  has three known   periodic
orbits on $\NN^{+}$, which are 
$$
\{ 1\},  \quad \{ 5, 7, 10\} \quad  \mbox{and} \quad
\{ 17, 25, 37, 55, 82, 41, 61, 91, 136, 68, 34\}.
$$
The {\em $3x -1$ Conjecture} asserts that every integer 
$m \ge 1$ under iteration by $T_{-1}$ eventually enters one of these three periodic orbits. 
This conjecture also appears intractable at present.

More generally one may consider iteration of 
the {\em  $3x+k$ function}, where $k \equiv \pm1 ~(\bmod \, 6)$, given by 
\begin{equation}\label{501}
T_k(n) := \begin{cases}
\frac{3n+k}{2} & \text{if $n$ is odd};\\
\frac{n}{2} & \text{if $n$ is even}.
\end{cases}
\end{equation}
The $3x+k$  functions were studied  in \cite{Lag90}, 
in connection with 
rational cycles for the $3x+1$ function.
Those periodic orbits of the $3x \pm k$ function 
 restricted to  the domain of all
integers $n$ having $(n, k)=1$ are known to 
correspond to those rational cycles for the $3x+1$ problem whose
members  each have  denominator $k$,  when written in lowest terms.

We  let $T^{\circ j}(m)$ denote the $j$-th iterate of a map $T:\ZZ \to \ZZ $, and denote the
{\em forward orbit } of $m$ by 
$$
\sO_{k}^{+}(m)  := \{ n: \, n=T_k^{\circ j}(m) \, \mbox{for some} \,  j \ge 0\}.
$$
In terms of forward orbits the  $3x+1$ Conjecture asserts that $1 \in \sO^{+}(m)$ for each integer $m \ge 1$.
Additionally we  define the {\em backward orbit} (or {\em inverse orbit}) of $m$ by 
$$
\sO_{k}^{-}(m)  := \{ n : \, T_k^{\circ j}(n):=m, \, \mbox{for some} \, \, j \ge 0\}.
$$
The set $\sO_K^{-}(m)$  comprises the forward orbit of $m$ under the (multivalued) inverse map 
\begin{equation}\label{101b}
T^{\circ-1}(n) = T_1^{\circ -1}(n) := 
\begin{cases}
\{ 2n\} & \text{if} \, n\equiv 0, 1 \, (\bmod \, 3);\\
\{ 2n, \frac{2n-1}{3}\} & \text{if} \,  n \equiv \, 2 \, (\bmod\, 3).
\end{cases}
\end{equation}
The  $3x+1$ Conjecture formulated in terms of backwards orbits asserts that $\sO_{1}^{-}(1) = \NN^{+}$,
where $\NN^{+}$ denotes the set of positive integers.

The main objects of study of this paper are  the 
{\em backward orbit generating functions}
\begin{equation}\label{eq14}
f_{k, m}(z) := \sum_{n \in \sO_{k}^{-}(m)\cap \NN^{+} } z^n,
\end{equation}
with  $k \equiv \pm \, 1 \, (\bmod \, 6)$
and $m \in \ZZ$.
The  functions $f_{k, m}(z)$ are  analytic functions of $z$ in the open unit disk $\{ z\in \CC: |z| <1\},$
and we consider the problem of when these generating functions are analytically continuable to
 larger domains in the complex plane.

Our main result, Theorem \ref{th10a} below,
formulates  conditions  characterizing for the $3x+k$ problem when the  
 generating function of a finite union of backwards orbits 
is a rational function of $z$. 
There is a known dichotomy for analytic continuation of a class of functions including the type above: they
either have the unit circle as a natural boundary to analytic continuation or else are  rational functions   (the P\'{o}lya-Carlson theorem).
Using this dichotomy we  deduce that 
 for each $k \in \ZZ$ with $(k, 6)=1$ and for almost all $m \ge 1$ the functions $f_{k, m}(z)$
have the unit circle $\{ |z|=1\}$ as  a natural boundary to analytic continuation. 

The functions $f_{k,m}(z)$  encode data  on the orbit intersected with  the positive
integers. However  with the proper choice  two such functions one can cover the orbit
on the negative integers as well.   To see this, we first note  that the  $3x+k$ function and the $3x-k$ function are conjugate under the
involution  $J: \ZZ \to \ZZ$ with  $J(x)= -x$,   i.e. 
$$T_{k} \circ J = J \circ T_{-k},$$
an operation which exchanges positive and negative
integers.
Consequently the behavior  of the function $T_k$ restricted to the negative integers
exactly matches that  of $T_{-k}$ restricted to  the positive integers, with the latter encoded 
by the generating function above. Thus the two generating functions $f_{k, m}(z)$ and $f_{-k, -m}(z)$
between them give  full information\footnote{For  functions $T_{k}$ that we consider no orbit contains $m=0$ except for the single point orbit  $\{0\}$.}
 for the inverse orbit $\sO_k^{-1}(m)$ on $\mathbb{Z}$.

For  the special case $k =\pm 1$ corresponding to the $3x+1$ function
and the $3x-1$ function respectively, the  positive integers $\NN^{+}$ and 
negative integers $\NN^{-}= - \NN^{+}$ are each bi-invariant sets for $T_k$; i.e., they are closed under forward and backward iteration.
The conjugacy function $J$ above shows that the
  $3x-1$ function on $\NN^{+}$ has iterates matching those of the $3x+1$ function
on the negative integers. 
In this special case we need only consider  $m \ge 1$, using $f_{1, m}(z)$ and $f_{-1, m}(z)$.
The  generating functions $f_{1, m} (z) = f_{-1, m}(z) \equiv 0$ for all $m \le -1$.


\subsection{Main results}

The  results of this paper concern properties of generating functions for the
 set union of a finite number of backward orbits of the $3x+k$ map.
 
 We first observe that
the backward orbits of a general  function $T: X \to X$ have  a trichotomy of possible behaviors: 
Two distinct such orbits $\sO_{T}^{-}(m_1)$ and $\sO_{T}^{-}(m_2)$  either  are disjoint or have one 
of them  properly contained in the other. These three outcomes correspond to one of: 
\begin{enumerate}
\item[(i)]  $m_1 \in \sO_{T}^{-}(m_2)$;
\item[(ii)] $m_2 \in \sO_{T}^{-}(m_1)$;  
\item[(iii)] neither (i) nor (ii) holds.
\end{enumerate} 
It follows from this trichotomy that for any  map $T$ on a countable set $X$, 
the set union $\sS$ of any  finite collection of backward orbits
can always be partitioned into a set of  {\em disjoint}  backward orbits of single elements, whose  set union equals that of the whole
collection.

Our main result characterizes when a finite union of backwards orbits has  generating function
that is a rational function. 


\begin{theorem} \label{th10a}
Consider the $3x+k$ map $T_k$ for an integer $k \equiv  \pm 1\, (\bmod \,6)$.
The following two conditions on a set union $\sS = \bigcup_{i=1}^{\ell} \sO_{k}^{-} (m_i)$
 of  a finite set of backward orbits 
$\{ \sO_{k}^{-} (m_i); 1 \le i \le \ell\}$ of $T_k$ are equivalent.

\begin{enumerate}
\item[(1)] The generating function  of  $\sS$
 restricted to $\NN^{+}$, which is
$$
g(z) := \sum_{n \in \sS \cap \NN^{+}} z^n,
$$
 is a  rational function of $z$.
\item[(2)] There is 
a set $X$ of  residue classes $(\bmod\, |k|)$ and a positive integer $k_0$ such  that the rational function
$$
h(z) = \sum_{{n>0}\atop{ n (\bmod \, |k|) \,\in X}} z^n = \sum_{{a \in X}\atop{1 \le a \le |k|}} \frac{z^a}{1- z^{|k|}},
$$
has power series coefficients agreeing with $g(z)$ for all $n \ge k_0$, so that $g(z)-h(z)$ is a polynomial of degree at most $k_0-1$.
That is,  the set of all $n \ge k_0$  belonging to $\sS$ contains exactly  
those $n \ge k_0$ that belong to the union of the arithmetic progressions $(\bmod\, |k|)$\\ in $X$.
\end{enumerate} 
If the  equivalent conditions (1), (2) hold, 
then the set $X$ 
 of residue classes in (2) is  closed under the action of the maps $r \mapsto 2r$ and $r \mapsto 3 r$
acting on residue classes $(\bmod \, |k|)$. 
 \end{theorem}

We prove Theorem \ref{th10a}  in Section \ref{sec3}. 
The proof uses the Skolem-Mahler-Lech 
Theorem, whose statement we recall in  Section \ref{sec2},
together with the trichotomy above.
The logical status of this result is interesting:  
At present we do  not know of  a single  case  of a value $k$ and a set  union $\sS$ where
either of conditions (1) or (2) hold unconditionally.
However if  the $3x+1$ conjecture
is true then  for $k=1$
there exist infinitely many  examples of finite unions of backward orbits $\sS$
where conditions (1)--(2) hold. 
See Section \ref{sec6} for discussion of the case of general $k \equiv \pm 1~(\bmod \, 6)$.

We deduce  from Theorem \ref{th10a} several 
consequences  about backward orbit generating functions having natural boundaries,
given in Theorems \ref{th11} - \ref{th14} below.

First we  consider natural boundaries for the  $3x+1$ function. 

\begin{theorem}\label{th11}
Consider  the $3x+1$ map $T_1$ on the positive integers $\NN^{+}$. For the  inverse orbit
generating functions $f_{1,m}(z) = \sum_{ n \in \sO_1^{-}(m)} z^n$ with starting value $m \ge 1$ the following hold.

(1) For each $m\ge 1$ except possibly $m=1, 2, 4$ and $8$ the generating function $f_{1,m}(z)$
has the unit circle $\{ |z|=1\}$ as a natural boundary to analytic continuation.

(2) If the $3x+1$ Conjecture is true, then  for \,$m=1, 2, 4$ and $8$
the generating function  $f_{1, m}(z)$
analytically continues to a rational function of $z$. If the $3x+1$ Conjecture is false,
then each of these four functions has  the unit circle $\{ |z|=1\}$ as a natural boundary to analytic continuation.
\end{theorem}

Secondly we  consider natural boundaries for the $3x-1$ function.

\begin{theorem}\label{th12}
Consider the $3x-1$ map $T_{-1}$ on the positive integers $\NN^{+}$.
For every starting value $m \ge 1$ 
the  backward orbit
generating function
$$f_{-1,m}(z) = \sum_{ n \in \sO_{-1}^{-}(m)} z^n$$ 
has the unit circle $\{ |z|=1\}$ as a natural boundary to analytic continuation.
\end{theorem}

 This result proves a conjecture of Berg and Opfer \cite[Conjecture  2.4]{BO13},
 which concerns analytic continuability of the three functions $\eta_1 (z) = f_{-1, 1}(z), \eta_2(z) = f_{-1, 5}(z)$
 and $\eta_3(z) = f_{-1, 17}(z)$.
 Theorems \ref{th11} and \ref{th12} are  proved in Section \ref{sec4} and make use of the P\'{o}lya-Carlson theorem
 given in Section \ref{sec2}.

Finally, in Section \ref{sec5} we establish  an analogous  result for the general case $k \equiv \pm 1 \, (\bmod\, 6)$,
which is  less specific about exceptional cases.

\begin{theorem}\label{th14}
Consider the $3x+k $ map $T_{k}$ with $k \equiv \pm 1 \, (\bmod \, 6)$ on the positive integers $\NN^{+}$.
Then  for all but finitely many starting values  $m \ge 1$ the 
 backward orbit generating
function $f_{k, m}(z) $ 
has the unit circle $\{|z|=1\}$ as a natural boundary to analytic continuation.
\end{theorem}

Establishing Theorem \ref{th14}  involves two additional difficulties.
First, the  orbits of $T_k$  for general $k$ are more complicated than those for $k =\pm 1$;   
they  can have several 
collections of residue classes $(\bmod \, |k|)$ that are  forward-and-backward invariant sets for the general $3x +k$-function 
as given in Lemma \ref{lem51}. (There is only one such class for $k = \pm 1$.)
Second,  some orbits may simultaneously contain infinitely 
many positive integers and   infinitely negative integers.  
This latter fact necessitated 
our  definition of  orbit generating functions  in \eqref{eq14}  to require  intersecting the orbit with the set of 
positive integers.


\subsection{Complexity of backwards orbits in the $3x+1$ problem}\label{sec12}

The $3x+1$ Conjecture asserts that the backward orbit $\sO_{1}^{-}(1)$
has a simple form.
However the simplicity of form of this statement  
seems to be an illusion that hides the difficulty of the problem.
The results above show that most  backward orbits $\sO_{1}^{-}(m)$ for variable $m$
 have considerable complexity; that is,  they cannot have the regular structure
required for the backward orbit generating function to be a rational function.

Indeed, the backward orbits of the $3x+k$ maps 
for different $k$ appear to have an extremely complicated structure.
The complexity of backward orbits was  noted already for the  $3x-1$ problem
by  Berg and Opfer \cite[Figure 1]{BO13}. They presented numerical data of membership in the three disjoint backward orbits
$\sO_{-1}^{-}(1),  \sO_{-1}^{-}(5),\sO_{-1}^{-}(17)$ in blocks of consecutive integers.
Each of these three orbits appears experimentally 
 to contain  a positive density of integers, and  the members of the 
different backward orbits appear to interlace 
in a complicated way on these blocks of integers. 
It would be interesting to formulate and study statistics which measure the amount
of intertwining complexity between two disjoint backward orbits of a fixed function $T_k$.


\subsection{Effective Computability Issues}
The computational problem of deciding whether, given input data $(k, m)$, the
function  $f_{k,m}(z)$ has $\{ |z| =1\}$ as a natural boundary
 is not known to be effectively computable.  
There does  exist a computational algorithm\footnote{By an algorithm we mean a procedure which can be
programmed on a Turing machine.  Such a procedure, however, is not certified to halt on all inputs.}  which, given $k$ as input,  {\em if it halts,}
will list a finite exceptional set $E_k \subset \ZZ$ and a proof that for all integers $m$ not in $E_k$,
the generating function $f_{k, m}(z)$ has the unit circle as a natural boundary
to analytic continuation.  This algorithm is described at the end of Section \ref{sec5}.
Standard conjectures analogous to the $3x+1$
Conjecture would imply that this algorithm will always halt.
At present it remains an open problem to prove (or disprove) that this algorithm always halts. 

There are two obstacles to
 proving effective computability for individual inputs $(k, m)$.
 \begin{enumerate}
 \item
 There is no effective algorithm  known which when
given $m_1, m_2$ will  determine whether 
the backward orbits  of $m_1$ and $m_2$ under $T_k$ are disjoint. 
\item
It is not  known whether each bi-invariant component of the $3x+k$ map
contains a finite cycle, although conjecturally this is always the case (\cite[Sect. 3.1]{Lag90}).
\end{enumerate}
Each of these obstacles appears to be an intractable problem at present.

\subsection{Contents of paper.}

In Section \ref{sec2} we recall several well known results on the
structure of power series with integer coefficients being rational functions
or having a natural boundary to analytic continuation.
In Section \ref{sec3} we prove Theorem \ref{th10a}; in Section \ref{sec4} 
 we prove Theorems \ref{th11} and  \ref{th12}, and  in Section \ref{sec5} we prove
 Theorem \ref{th14}.

In Section \ref{sec6} we complement Theorem \ref{th11}
by showing  that standard conjectures on the iteration of
the $3x+k$ maps imply that there will exist infinitely many cases where a
a finite sum of backward orbit generating functions is a rational function.

In the concluding section  we describe other known results for
different analytic generating functions associated to $3x+k$ mappings which are
(or may sometimes be) rational functions. 

%
%
%
%

\section{Preliminary Facts}\label{sec2}

The proofs of the paper use several well-known 
results in analytic function theory.

\subsection{Fabry Gap Theorem}

The  following  basic result on natural boundaries of analytic
functions given by lacunary expansion can be applied to  very special cases 
of backwards orbits above. 
\begin{theorem} {(\em Fabry Gap Theorem)}\label{th31}
Let $f(z) = \sum_{j=0}^{\infty} a_j z^{n_j}$ be a power series  a positive finite radius of convergence $R$,
and have gaps between its exponents $n_j$ in the sense that 
$$
\lim_{j \to \infty} \frac{n_j}{j} = +\infty.
$$ 
 Then the disk $B_R(0) = \{ |z| < R\}$
is a maximal domain of holomorphy for $f(z)$, i.e. the circle $|z|=R$ is a natural boundary to
analytic continuation.
\end{theorem}

The original result of Fabry \cite{Fabry1896} in 1896 proved this theorem under the weaker assumption that
the exponents are lacunary in the sense that there is a constant $C>1$
such that $n_{k+1}/n_k > C$ for all sufficiently large $k$.
The stronger result stated above is due to Faber \cite{Fab06} in 1906,
see Remmert \cite[p. 256]{Rem98}. There are now much stronger gap theorem results  known, 
showing that if the average gap size tends to infinity then there is a natural
boundary, see  \cite[p. 256]{Rem98}.

Theorem \ref{th31} has an immediate application to the $3x+k$ problem based
on the observation that 
for any $k \equiv \pm 1~(\bmod \, 6)$. the backwards orbit of any $m \equiv 0 \, (\bmod \, 3)$ is
 $ \sO_{k}^{-1}(m) = \{ 2^k m:  ~\, k \ge 0 \}.$
 It follows that for  any positive $m \equiv 0 \, (\bmod \, 3)$  the generating function
 of its inverse orbit 
 $$
 f_{k, m}(z) = \sum_{n \in \sO_{k}^{-}(m)} z^n  = \sum_{n=0}^{\infty}  z^{ m 2^n}
 $$
  satisfies the hypotheses of the Fabry gap theorem, so  has
 $\{|z|=1\}$ as a natural boundary. 

 For the $3x+1$ function $T_1$ it has been conjectured that 
for $m \not\equiv 0 ~(\bmod \, 3)$ the expected number of inverse iterates below $x$
is asymptotically bounded below by $c(m)x$ for a positive constant, see \cite[Conjecture A]{AL95a}.
If such a conjecture were true, then the existing gap-type theorems do not apply
to infer  the existence of a natural boundary to analytic continuation. 
The  main results of this paper are proved without appeal to  the Fabry Gap Theorem. 

\subsection{P\'{o}lya-Carlson Theorem}

The following dichotomy theorem was conjectured by P\'{o}lya \cite{Pol16} in 1915, and proved by
Carlson \cite{Car21} in 1921. 
\begin{theorem} {(\em P\'{o}lya-Carlson Theorem)}\label{th32}
Let $f(z) = \sum_{n=0}^{\infty} a_n z^{n}$ have integer coefficients and have
radius of convergence $R=1$. Then exactly one of the following holds.
\begin{enumerate}
\item[(i)]
The power series  $f(z)$ has the circle $\{|z|=1\}$ as
a natural boundary to analytic continuation;
\item [(ii)]
The power series  $f(z)$ can be analytically continued
to a rational function of the form $\frac{p(z)}{(1-z^m)^n}$ for a polynomial $p(z) \in \ZZ[z]$
with $m, n$ positive integers.
\end{enumerate}
\end{theorem}

 This results has since been strengthened  and extended in various ways,
see Remmert \cite[p. 265]{Rem98} and Bell, Coons and Rowland \cite{BCR12}.

\subsection{Skolem-Mahler-Lech theorem}

The Skolem-Mahler-Lech theorem has many different incarnations,
for which see van der Poorten \cite{vdP89} and Everest et al \cite[Chap. 2]{EPSW03}.
It can be formulated in terms of zeros of recurrence sequences or zeros of Taylor series
coefficients of rational functions; for more recent algebraic geometric versions, see \cite{BL14}. 
We will use the following version. 

\begin{theorem} {\rm (Skolem-Mahler-Lech theorem)}\label{th33}
Let $R(z) \in \CC(z)$ be a rational function. Suppose that 
$R(z)$ is holomorphic at $z=0$ 
and let its Taylor series
expansion around $z=0$ be 
$$
R(z) = \sum_{n=0}^{\infty} c_n z^n.
$$
Then the set  of indices of vanishing Taylor coefficients 
$Z(R; z=0) := \{n:  c_n =0\}$ 
can be partitioned into a finite (possibly empty)
set of complete half-infinite arithmetic progressions
$P(a; d)^{+} := \{ n: n \equiv a~(\bmod\, d)\, \mbox{with} \,\, n>0\}$, 
for some modulus $d$, up to the inclusion or exclusion of a finite set.
\end{theorem}

\begin{proof} This result is a consequence of two results given  in  \cite{EPSW03}.
The first is  a result showing  that the  Taylor coefficients of a rational function $R(z)$
satisfy a linear recurrence with constant coefficients, with a converse stating that  any Taylor expansion having
this property is the Taylor expansion of a rational function
(\cite[Theorem 1.5]{EPSW03}). The second result is  
the recurrence sequence form of the Skolem-Mahler-Lech
theorem (\cite[Theorem 2.1]{EPSW03}). For a detailed proof of the latter result, see \cite{vdP89}.
\end{proof}

Rather remarkably, the known proofs of the Skolem-Mahler-Lech theorem use $p$-adic methods
for some suitably chosen prime $p$, despite this theorem statement being
 an assertion over the complex numbers.
 
At present there is no general algorithm known to effectively determine the
set of zero coefficients of a rational function. 
 There does exist an effectively computable algorithm to determine a set of
arithmetic progressions satisfying the conclusion of the theorem, and in particular, there is
an effectively computable algorithm to determine
 whether the exceptional set $Z(R; z=0)$ is finite or infinite.  Evertse, Shlickewei
and Schmidt \cite{ESS02} establish an effective upper bound on the 
number of exceptional zeros in the case where $Z(R; z=0)$ is finite; they 
do not, however, obtain a bound on the size of such zeros.  

%
%
%
%

\section{backward Orbits of $3x+k$ maps and the SML Property }\label{sec3}

It is convenient to restate the conclusion of the Skolem-Mahler-Lech theorem as
asserting a general  property of sequences.

\begin{defi} \label{de31}
We say that a  sequence $\{ c_n: n \ge 0\}$ of complex numbers has  the {\em SML property}
if the set $\{ n \ge 0: c_n =0\}$ can be partitioned into a finite
set of  half-infinite arithmetic progressions
\begin{equation}\label{ap}
P^{+} (a; d):= \{ n: n \equiv a~(\bmod\, d)\, \mbox{with} \,\, n>0\}
\end{equation}
plus a finite set (possibly empty).
\end{defi}

 We now  prove Theorem \ref{th10a}.
 
 \begin{proof}[Proof of Theorem \ref{th10a}.]
We set $\sO := \Big( \bigcup_{i=1}^{\ell} \sO_{k}^{-}(m_i) \Big) \cap \NN^{+}.$
By the remark just before Theorem \ref{th10a}, we may
assume without loss of generality that the $\sO_{k}^{-}$ are pairwise disjoint sets.

  Suppose first that property (2) holds. Then letting the members of $X$ be
 the least nonnegative residues in the congruence classes $(\bmod \, |k|)$ we may write
 $$
 g(z) := \sum_{ n \in \sO} z^n = \sum_{j=0}^{k_0 -1} a_i z^i + \sum_{j \in X} \frac{z^{k_0 + j}}{1- z^{|k|}},
 $$
 which certifies that $g(z)$ is  a rational function, so property (1) holds.
 
  Now suppose property (1) holds. Then by hypothesis  
 $$
 g(z) = \sum_{j=1}^{\ell} f_{k, m_j}(z)= \sum_{j=1}^{\infty} c_n z^n
 $$ is a rational function, 
 so the Skolem-Mahler-Lech theorem applies to show that
  $$
S:= \{ m: c_m =0\} := \NN^{+} \smallsetminus \sO
 $$
 has the SML property, i.e.  
 its members are eventually periodic modulo some finite modulus $d \ge 1$, i.e
 \begin{equation}\label{admis}
 n \in S \Rightarrow   \,\, n+d \in S  \quad \mbox{for all} \quad n \ge n_0(d).
 \end{equation}
That is, $S$  coincides with the set union of a finite set of complete  arithmetic progressions 
$$
P(a_i; d) := \{ n \ge 1:  n \equiv a (\bmod \, d) \},
$$
 up to the inclusion or exclusion of a  finite set of exceptional values.
We call a modulus $d$ {\em admissible}  if it has the eventual periodicity 
property \eqref{admis}. Here  $S$ is admissible if and only if its  complement $S^c := \NN^{+} \smallsetminus S= \sO$
is admissible. 
  
  We first consider the minimal admissible  $d$ for which such a property holds. Now we have 
 a partition of all residue classes $(\bmod \, d)$  given by $X \cup Y$ with 
 \begin{eqnarray*}
 X & := & \{ i~(\bmod\, d) : P(i; d) \cap S^c \quad \mbox{is finite}\} \\
 Y   &:= & \{ i~(\bmod\, d) : P(i; d) \cap S \quad \mbox{is finite}\}.
 \end{eqnarray*}
 In particular, if a coset $j\, (\bmod \, d)$ contains infinitely many elements of $\sO$
 then this coset must belong to $X$.

\begin{claim} \label{claim1}
The minimal admissible modulus $d$ is odd.
\end{claim}

\begin{proof}
Suppose to the contrary that $d= 2d'$ were even. By minimality there must exist 
a residue class $i ~(\bmod \, d)$ such that  $i \, (\bmod \, d) \, \in X$ while  $i+d' (\bmod \, d)$ in  $Y$.

We assert that the residue class $ 2i~(\bmod \, d) \in X$. Pick a modulus $m_j$ such that there are
infinitely many
inverse iterates of $m$ in the class $i\, (\bmod \, d)$.
Now for  $n \in X$  with $n \in \sO_{k}^{-} (m_i)$ then $2n \in \sO_k^{-}(m_i)$. 
This gives infinitely many elements of $X$ in  the residue class $2i ~(\bmod \, d)$, whence this entire residue class is in $X$.
We conclude  that all sufficiently large members of $P(2i; d)$ belong to $X$.  

However 
$$
T_k(2i + dn) = T_k(2 (i+d'n)) = i+ d'n,
$$
hence half of the elements $P(2i; d)$, those with $n$ odd,  arise as preimages under $T_k$ of the residue class $P( i+ d'; d)$.
Since all sufficiently large  elements of $X$ are the preimage under $T_k$ of some other element of $X$, we conclude that
infinitely many elements of $X$ fall into the class $i + d' ~(\bmod \, d)$, whence
this class also belongs to $X$. But this contradicts the hypothesis that this class belonged to  $Y$,
and Claim \ref{claim1}  follows.
\end{proof}

\begin{claim} \label{claim2}
The minimal admissible modulus $d$ is not divisible by $3$.
\end{claim}

\begin{proof}
Suppose to the contrary that $d= 3d'$ were divisible by $3$.
Then there must exist  a residue class  $i ~(\bmod\, d) \in X$,
with at least one of $i+ d' \, (\bmod \, d)$ and $i+2d'\, (\bmod \, d)$ in $Y$.
By Claim \ref{claim1} $d$ is odd, so by replacing $i$ by $i+d$ if
necessary we may assume that $i$ is odd. 
The residue class $i \, (\bmod \, d)$ contains all sufficiently large elements of
form $n= i+ 2dn'$  in $X$,  and since these are odd numbers, 
with at most $\ell$ exceptions 
$$
T_k(n) = \frac{3i+k}{2} + 3dn'  \in X.
$$
This exhibits infinitely many numbers   $ \frac{3i+k}{2} \,(\bmod \, d)$ in $X$,
so it follows that  $\frac{3i+k}{2} \, (\bmod \, d)$ belongs to $X$. 
However one has
$$
T_k(i + 2d' + 2dn' ) = \frac{3i+k}{2} +  d(3n'+1)  \subset P( (3i+k)/2;  d)
$$
and also
$$
T_k( i + d'  +(d+ 2dn')) = \frac{3i+k}{2}  + d(3n' +1) \subset P( (3i+k)/2;  d).
$$
Thus each residue class $i + d' (\bmod \, d)$ and $i+2d' (\bmod \, d)$ contains
infinitely many elements of $X$, so both these residue classes
must belong to $X$. This contradicts the fact that one of these classes is in $Y$
and Claim \ref{claim2} follows.
\end{proof}

\begin{claim} \label{claim3}
The modulus  $d=|k|$ is an admissible modulus.
\end{claim}

\begin{proof}
Since we now know the minimum admissible modulus $d$ has $(d, 6)=1$, by splitting into  smaller residue
classes as necessary we may obtain an admissible modulus $d'$ such that  $|k|$ divides $d'$,
and $(d', 6)=1$. 
We now redefine  $d \ge 1$ to  be the minimal
admissible modulus having the property that $|k|$ divides $d$.
The claim asserts that   $d= |k|$.

Since $(d, 6)=1$, both $2$ and $3$ are invertible $(\bmod \, d)$. 
By applying forward iteration, since $n \in \sO_{k}^{-}(m_j)$ implies $T_k(n) \in \sO_{k}^{-} (m_j)$
with at most one exception, and since each residue class $(\bmod \, d)$ contains
infinitely many even integers and infinitely many odd integers, we conclude that 
$$
i ~(\bmod \, d) \in X  \Rightarrow \, \frac{i}{2} ~\, (\bmod \, d) \in X.
$$
Since each congruence class $(\bmod \, d)$ contains infinitely many odd numbers,
we also have
$$
i ~(\bmod \, d) \in X  \Rightarrow \, \frac{3i +k}{2} ~\, (\bmod \, d) \in X.
$$
By applying a single step of backward iteration of $T_k$, and observing each residue class $(\bmod \, d)$ has
infinitely many integers in each residue class $(\bmod \, 6)$ we obtain
$$
i ~(\bmod \, d) \in X \Rightarrow \, 2i ~\, (\bmod \, d) \in X.
$$
$$
i ~(\bmod \, d) \in X  \Rightarrow \, \frac{2i-k}{3}  ~\, (\bmod \, d) \in X.
$$
These results show that the set of residue classes $(\bmod \, d)$
in $X$ is closed under the action of the transfomations  $S_1(r) = 2r$
and  $S_2(r) = \frac{3r+k}{2}$ as well as under their inverses
$(S_1)^{-1} (r) = \frac{r}{2}$ and $(S_2)^{-1}(r) = \frac{2r-k}{3}$.  Now set
$S_3 (r) := S_1 \circ S_2 (r) = 3r +k$, with inverse map $(S_3)^{-1}(r) = \frac{r-k}{3}.$
One calculates the commutator map 
$$
S_1 S_3 (S_1)^{-1} (S_3)^{-1}(r) = r+ k.
$$
and
$$
S_3 S_1 (S_3)^{-1} (S_1)^{-1}(r) = r-k,
$$
We deduce that for all sufficiently large members of $X$, 
$$
n \in X  \Rightarrow n+|k| \in X.
$$
This fact shows that we may choose the modulus $d= |k|$, proving Claim 3.
\end{proof}

Claim 3 now  establishes that the set of residue classes in both $X$ $(\bmod \, |k|)$
is  invariant under the action of the maps $x \mapsto 2x$ and $x \mapsto 3x+k  \equiv  3x ~(\bmod \, |k|)$.
It follows that the complementary set $Y$ of residues is also invariant under these maps,
which are invertible since $(k, 6)=1$. This shows that property (2) holds, 
and also gives the extra invariance condition on the residue classes on $X$. 
 \end{proof}

%
%
%
%

\section{Natural Boundaries for Backward Orbits of $3x\pm 1$ functions}\label{sec4}

 We use Theorem \ref{th10a} to show  the existence of  natural boundaries for
the generating functions of most backward orbits for the $3x \pm 1$ functions,
as stated  in Theorems \ref{th11} and \ref{th12}.
 
 \begin{proof}[Proof of Theorem \ref{th11}]
 
 (1) By the P\'{o}lya-Carlson Theorem, for $\sO_{-}(m)$ 
 the associated generating function $f_{1, m}(z)$ will have $\{|z|=1\}$
 as a natural boundary if and only if $f_{1, m}(z)$ is not a rational function.
 By  Theorem ~\ref{th10a}, the orbit generating function  $f_{1, m}(z)$ will
 be a rational function if and only if the backward orbit $\sO_{1}^{-}(m)$ is eventually periodic modulo $1$,
 i.e. it contains all sufficiently large integers $n$.  
 To show a given backward orbit $\sO_{1}^{-}(m)$ for $m \ge 1$ has generating function $f_{1, m}(z)$ that
 is not a rational function, it suffices to show  there is another backward orbit  $\sO_{1}(m')$  disjoint from it
 for some $m' \ge 1$. 
 
 If the $3x+1$ Conjecture is false, then there exists a backward orbit $\sO_{1}^{-}(m')$ of positive
 integers disjoint from  $\sO_{1}^{-}(1)$, and  these orbits are infinite. It follows that 
 $f_{1, m}(z)$ is not a rational function
 of $z$ for all $m \ge 1$, since the backward orbit $\sO^{-1}(m)$ must be disjoint from at least one of
 the orbits $\sO_{1}^{-}(1)$ or $\sO^{-}(m)$. 
 
 If the $3x+1$ Conjecture is true, then
 for $m = 1, 2, 4, 8$  we have
 $$
 f_{1, m}(z) = (\sum_{n=1}^{\infty} z^n) - p_m(z)  = \frac{z}{1-z}- p_{m}(z),
 $$
in which  $p_m(z) = 0, z, z + z^2, z+ z^2+z^4$, and all these $f_{1,m}(z)$ are
rational functions. All other $m \in \NN^{-}$ then belong 
to one of  $m \in \sO_{1}^{-}(5)$ or $m \in \sO_{1}^{-}(16)$. These two backward orbits are  infinite
 and are disjoint, which certifies that the elements $\sO_{1}^{-}(m)$ for such $m$ 
 are eventually periodic modulo $1$, whence the contrapositive of Theorem \ref{th10a}
 implies $f_{1, m}(z)$ is not a rational function.  
    \end{proof}

\begin{proof}[Proof of Theorem \ref{th12}]
 For the $3x-1$ function,  it is known that $\sO_{-1}^{-}(1)$, $\sO_{-1}^{-}(5)$ and
 $\sO_{-1}^{-} (17)$ are infinite disjoint sets.  It follows that for any $m \ge 1$ 
 the  elements in the backward orbit  $\sO_{-1}^{-}(m)$ cannot  be eventually periodic
 modulo $1$. The contrapositive of Theorem \ref{th10a} for $k=-1$ implies that
 $f_{-1, m}(z)$ is not a rational function. By the P\'{o}lya-Carlson theorem it then must
 have the unit circle $\{ |z|=1\}$ as a natural boundary to analytic continuation. 
 \end{proof}
%
%
%
%

\section{Natural Boundaries for Backward Orbits of $3x+k$ functions}\label{sec5}

We first discuss   complications in the iteration of general $3x+k$ maps compared
with the $3x\pm 1$ maps, and then prove Theorem \ref{th14}.

The behavior of general $3x+k$ maps under iteration exhibit three 
features not occurring for   $3x\pm 1$ maps. 
The first of these additional features is that the domain $\ZZ$ splits into 
various  bi-invariant sets for $T_k$, as follows. 
\begin{lemma} \label{lem51}
For $k \equiv \pm 1 ~(\bmod \, 6)$, partition the congruence classes $~(\bmod \, |k|)$ into
bi-invariant sets under the action of the group generated by multiplication by $2$ and $3$
on $(\ZZ/k\ZZ)^{*}$.  Then for $1 \le a \le |k|$ each of  the sets 
$$X_{a, k} := \bigcup_{i, j \ge 1}  P( 2^i 3^j a;  |k|) $$
is forward invariant and backward invariant 
under $T_k$.
\end{lemma}

\begin{proof}
Using $(|k|, 6)=1$ 
the two  transformations acting
on the finite group $\ZZ/ k \ZZ$  given by $S_1(r)= 2r$ and $S_2(r) = \frac{3r+k}{2}$
are invertible and so form a group.  The proof of Claim 3 of Theorem \ref{th11}
shows that this group 
is generated by $S_1(r)$ and $S_3(r) = 3r$. The 
projections on $\ZZ/k \ZZ$ of the sets $X_{a, k}$ are minimal
sets closed under the action of these generators. 
The forward and backward invariance of
these  sets immediately follows from this fact.
\end{proof}

The second additional feature  of general $3x+k$ maps is that for  positive divisors $d$ of $k$, each of the sets
$$
S_{d,k} := \bigcup_{a: (a, |k|)= d} P(a; |k|)
$$
is bi-invariant set for $T_k$, and is a disjoint unions of classes $X_{a, k}$.
 Furthermore the map $M_{d}$ of
multiplication by $d$ has
 $$
 M_{d} \Big(  S_{1, k/d} \Big) =  S_{d, k}. 
 $$ 
It defines a conjugacy map 
between $T_{k/d}$ and $T_k$ restricted to  these invariant sets,  i.e.
$$
M_{k/d} T_{k/d} (n) = T_k M_{k/d} (n) ~~~~\mbox{for all}~~ n \in S_{1, k/d}.
$$
This conjugacy also applies at the level of the smaller bi-invariant sets $X_{a, k}$, for
if $(a, k) =d$ then 
$$
M_{d} \Big(  X_{a/d, k/d}  \Big) =  X_{a,k} , 
$$

The third additional feature is that  $3x+k$ maps may   contain backward
orbits that have infinite intersection with both $\NN^{+}$
and with the negative integers $\NN^{-}$.
As an example of a backward  orbit unbounded for positive and negative
integers, take any $k>0$ with $k \equiv 5 ~(\bmod \, 6)$ and
the backward orbit  $\sO_{k}^{-}(1)$.
 We have $\frac{2-k}{3} \in \sO_k^{-}(1)$,
hence 
$$\{ (2-k)2^j, 2^j: j \ge 1\} \bigcup \{ 2^j : j \ge 1\} \subset \sO_k^{-}(1).$$
There are only finitely many such exceptional orbits,
and each of them contains an odd integer in the ``critical interval" $[-k, -1].$

\begin{proof}[Proof of Theorem \ref{th14}.]
We treat separately the action of $T_k$ on each bi-invariant subset $X_{a, k}$ of $\ZZ$
given by Lemma \ref{lem51}. There are finitely many such subsets,
so it suffices to prove  for each $X_{a, k}$ that for all but finitely many $m \in X_{a, k}$ with $m \ge 1$
 the generating function
$f_{k,m}(z)$ has $\{|z|=1\}$ as a natural boundary. By the  P\'{o}lya-Carlson theorem
this is equivalent to the assertion that $f_{k, m}(z)$ is a rational function for only finitely many $m \in X_{a, k}$ with $m \ge 1$. 
(Here the condition $m \ge 1$ must be imposed in the hypothesis of Theorem \ref{th14}
because there always exist infinitely many negative $m$ with $\sO_{k}^{-}(m) \subset \NN^{-}$
whence $f_{k,m}(z) =0$ is a rational function.)

 Now suppose that $m \in X_{a, k} \cap \NN^{+}$ is such that 
 $f_{k, m}(z)$ is a rational function. 
Then by   criterion (2)  of Theorem \ref{th11}, 
the backward orbit  $\sO_k^{-}(m)$ (which is  infinite)
 must contain all sufficiently
large positive integers in $X_{a, k}$.
However a backward
orbit $\sO_k^{-}(m)$ cannot have this  property  
whenever there exists another $m' \in X_{a, k}\cap \NN^{+}$ such that $\sO_k^{-}(m')$ is
disjoint from $\sO_{k}^{-}(m)$.
Since all backward orbits on positive
integers are infinite,  the finitely many elements $\Big(X_{a, k} \cap\NN^{+} \Big)\smallsetminus \sO_k^{-}(m)$
must by trichotomy have backward orbits intersecting $\sO_k^{-}(m)$, and therefore containing it.
By replacing $m$ with such an element, we may enlarge the backward orbit,
and thus in a finite number of steps arrive at a single positive element $m'$ whose backward orbit $\sO_k{-}(m)$ contains all of
$ X_{a, k}\cap \NN^{+}$. 
We now study this backward orbit;  note that it may contain some negative integers.
There are two cases to consider.\medskip

{\bf Case 1.} {\em The backward orbit $\sO^{-}(m')$ is a tree.}\medskip

This tree must necessarily branch at a lowest point where both subtrees at the
branch contain infinitely many positive integers,  
otherwise its density on the positive integers in $X_{a,k}$ would be
zero, contradicting that it covers all positive integers in $X_{a,k}$.
That is,  there can be at most a finite number of branchings where only
one of the two branch subtrees contain positive integers.
Go to the lowest such branching where both subtrees have positive
integers, hence infinitely many such integers. Let $m_1, m_2$ be
the lowest nodes in these two subtrees. Now $\sO_k^{-}(m_1)$ and $\sO_k^{-}(m_2)$
cover all but finitely many elements of $X_{a, k} \cap \NN^{+}$ and each certifies
that all nodes in the other have backward orbit generating functions having $\{|z|=1\}$
as a natural boundary.\medskip

{\bf Case 2.} {\em The backward orbit $\sO^{-}(m')$ contains a periodic orbit.}\medskip

In this case $m'$ itself must be in the periodic  orbit, and $\sO^{-}(m')$ consists of 
this periodic orbit plus a finite number of trees that enter it under forward iteration.  Let
$m_1, \ldots , m_k$ denote the lowest node in each such tree that does not belong to  the periodic
orbit.  If there are more than two such trees containing a positive element, then together
they certify that all positive elements $m''$ in all these trees have generating functions $f_{k, m''}(z)$ having $\{|z|=1\}$
as a natural boundary. In this case only the elements of the periodic orbit itself have
$f_{k, m''}(z)$ being a rational function. If, however, there is exactly  one tree entering
the periodic orbit, then we must go backward to the first branching node in this tree 
such that  both subtrees contain positive elements, and repeat the argument  in Case 1.
This latter case will occur for the $3x+1$ function if the $3x+1$ Conjecture is true,
taking $m'=1$ in that case.
\end{proof}

%
%
%
%

\section{Existence of Finite backward Orbits Covering Almost All Positive Integers}\label{sec6}

Standard conjectures for the $3x+k$ problem on the integers when $(k, 6)=1$ (see \cite{Lag90})
assert  that the following
hold.
\begin{enumerate}
\item[(i)]
({\em Finite Cycles Conjecture}) There are only finitely many periodic orbits of $T_k$ on
the integers $\ZZ$.
\item[(ii)]
({\em Divergent Trajectories Conjecture}) There are no divergent trajectories of $T_k$, i.e.
for no integer $m$ is its forward orbit $\sO_k^{+}(m)$ of infinite cardinality.
\end{enumerate}

The first conjecture follows from the Finite Primitive Cycles Conjecture made in \cite{Lag90},
taken over all the divisors of $k$.
The second conjecture is a generalization of the Divergent Trajectories Conjecture given
for the $3x+1$ problem in \cite[Sect. 2.7]{Lag85}; the heuristic argument justifying it applies just as
well to the $3x+k$ problem with $\gcd(k, 6)=1$.

We will show these two conjectures together  imply the existence of collections of orbits
satisfying the hypotheses (1)--(2) of 
 Theorem \ref{th10a}.  We make some preliminary definitions.
\begin{defi} \label{de61} 
 Two analytic functions $f_1(z), f_2(z)$ will be called {\em rationally equivalent}
if their difference $f_1(z) - f_2(z)$ is a rational function. Otherwise they are 
{\em rationally inequivalent}.
\end{defi}

\begin{defi} \label{de61} 
 Two collections of analytic functions $(f_1(z), f_2(z), \ldots , f_n(z))$ 
 and $(g_1(z), \ldots, g_m(z))$ are called {\em rationally equivalent} if there is
 a one-to-one correspondence of rationally equivalent pairs $(f_i(z), g_{\sigma(i)}(z))$
 where if the sequences are of unequal length one is padded with zeros to be
 the same length, and $\sigma$ is a permutation of indices  of that length. Otherwise they are 
{\em rationally inequivalent}.
\end{defi}

\begin{theorem} \label{th-over}
Suppose that the  Finite Cycles Conjecture and the Divergent Trajectories Conjecture
both hold for $T_k$. 
Then for each set $X$ of residue classes $(\bmod\, |k|)$ that is closed under the 
action of  the maps $r \mapsto 2r$ and $r \mapsto 3 r$
acting on residue classes $(\bmod \, |k|)$ 
there are infinitely many different finite collections  of disjoint backward orbits of form 
$\{ \sO_{k}^{-} (m_i); 1 \le i \le \ell\}$ (with all $m_i >0$) that are pairwise rationally inequivalent, such that:
\begin{enumerate}
\item
all orbits $\sO_{k}^{-} (m_i)$ take values only in residue classes in $X$;
\item
all sufficiently large positive integers in congruence classes in $X$
belong to their set union: $\sS :=\bigcup_{i=1}^{\ell} \sO_{k}^{-}(m_i).$
\end{enumerate}
In all these cases  the associated  generating function 
$g_{\sS}(z) := \sum_{n \in \sS \cap \NN^{+}} z^n$
is a  rational function.
\end{theorem}

\begin{proof}
The truth of the Divergent Trajectories Conjecture implies that each integer enters a periodic
orbit under forward iteration. The truth of the Finite Cycles Conjecture predicts there are finitely
many cycles, call their  generators $\{ m_i : 1 \le i \le \ell\}$, taking $m_i$ the element of smallest
absolute value in each cycle, making if positive if there is  a tie. 

We need only consider those backward orbits $\sO_{k}^{-}(m_i)$ that contain infinitely many positive
integers, which is the same as those that contain at least one positive integer.
The collection of these finite sets of backward orbits whose members $m_i$ belong to $X_{a,k}$
in Lemma \ref{lem51} will cover $X_{a, k} \cap \NN^{+}.$ They are disjoint backward orbits,
and 
$$
\sum_{m_i \in X_{a, k}} f_{k, m_i}(z)  = \sum_{n \in X_{a, k} \cap \NN^{+}} x^n,
$$
which is a rational function. Thus we obtain a finite sum of generating functions of elements
in each $X_{a, k}$ that is rational function. 

To get infinitely many distinct such identities we  trace back their trees of inverse iterates. 
 We  note  that at each step one has the set partition identity 
$$\sO^{-}(m) = \{ m \} \cup \left(\bigcup_{ \{m'\colon T_k(m') = m\}} \sO_k^{-}(m')\right).$$
This identity allows us to replace the generating function of the left side element  by the sum of generating functions of
the right side. We may then drop the generating function of the one element set,
and we get another  identity.  
Whenever a branching of the backward iterate trees  occurs we split the partition of $X_{a, k}$ into one larger set.
Some of these inverse iterate trees must infinitely branch, otherwise there will not be enough elements in
the inverse image to cover all elements of $ X_{a, k} \cap \NN^{+}$, so 
we obtain an infinite set of finite collections of backward 
orbit sets on $X= X_{a, k}$ for which property (2) holds.  

 In the process above, each  time a node is split we get a new collection of orbits whose associated collection of
  generating functions  are  rationally inequivalent to
each such set  that  was constructed earlier. Here we use the criterion of Theorem \ref{th11} to certify the two new functions are rationally inequivalent
to the function they replace and to each other, and to all other generating functions in the current partition.
\end{proof}

%
%
%
%

\section{Concluding Remarks}\label{sec7}

We describe some related problems 
for $3x+1$-type iterations which involve generating functions
with integer coefficients which are (or sometimes may be) rational functions.

Berg and Meinardus \cite{BM94} (see also \cite{BM95})
introduced a set of  generating functions for encoding information about
$3x+1$ iterates, given by power series with integer coefficients.
They introduced for each fixed $m \ge 1$ and for $k=1$ the function
\begin{equation}\label{eq71}
g_{k}^{(m)}(z) := \sum_{n=1}^{\infty} T_k^{\circ m} (n) z^n,
\end{equation}
which encodes the teration exactly $m$ times, 
where the input value varies.  One may call
these functions  {\em $m$-th iterate generating functions.}  These functions converge on the open
unit disk $\{ |z| < 1\}$.
Berg and Meinardus \cite[Theorem 2]{BM94} showed that for each $m \ge 1$ these  functions 
are rational functions of $z$,
and determined properties of these rational functions: all of their poles fall  on the circle $\{|z|=1\}$,
 comprise  a subset of the $2^m$-th roots of unity, and are at most double poles.
Chamberland \cite{Cha13} extended their results to   general $k \equiv \pm 1\, (\bmod \, 6)$,
and to  more general maps ($qx+k$ maps), and studied the structure of these rational functions
in more detail.

Berg and Meinardus \cite{BM94}
also  introduced for fixed $n \ge 1$ and $k=1$ the {\em forward orbit generating functions}
\begin{equation} \label{eq72}
h_{k, n}(w) := \sum_{m=0}^{\infty} T_k^{\circ m} (n) w^m.
\end{equation}
which encode the complete sequence of forward iterates of a fixed integer $n$.
These power series have integer coefficients, but their radii of convergence are
not known in general. 
For given $T_k$ and starting value $n \ge 1$, if the forward orbit of $n $ is eventually periodic 
then the power series for  $h_{k, n}(w)$   will
converge on $\{ |w|< 1\}$ and $h_{k, m}(w)$ will be  a rational 
function of $w$. In particular if the $3x+1$ Conjecture is true, then all $h_{1, n}(w)$ will
be rational functions. 
In the remaining case that   the forward orbit $\sO_{k}^{+}(n)$ is  a divergent trajectory,  Berg and Meinardus
 only assert that the radius of the disk on which the series converges
must be at least $\frac{2}{3}$. One might expect that $h_{k, n}(w)$ will  not be a rational function of $w$  in this case,
but justifying this expectation is   an open problem. 

Thirdly  Berg and Meinardus \cite{BM94} introduced for $k=1$  the bivariate generating functions
$$
F_k(z, w) :=\sum_{m=0}^{\infty} \sum_{n=0}^{\infty} T_k^{\circ m} (n) z^n w^m.
$$
They gave a system of functional equations which this generating function satisfies,
see \cite[Theorem 4]{BM94}. It is not known whether this function is ever a bivariate rational function.

%
%
%
%



\end{document}